\documentclass{amsart}
\usepackage{amsrefs}
\usepackage[ocgcolorlinks,linktoc=all]{hyperref}
\usepackage{lipsum}

\hypersetup{citecolor=blue,linkcolor=red}
\newtheorem{theorem}{Theorem}[section]

\newtheorem*{Cav}{Caveat}
\newtheorem{Qu}{Question}

\newtheorem{lemma}[theorem]{Lemma}

\newtheorem{corollary}[theorem]{Corollary}
\theoremstyle{definition}

\newtheorem{example}[theorem]{Example}

\theoremstyle{remark}
\newtheorem{remark}[theorem]{Remark}

\numberwithin{equation}{section}

\begin{document}
\title[The $L_p$-Christoffel-Minkowski problem]
 {Deforming a hypersurface by principal radii of curvature and support function}

\author[M.N. Ivaki]{Mohammad N. Ivaki}
\address{Department of Mathematics, University of Toronto, Ontario,
M5S 2E4, Canada}
\email{m.ivaki@utoronto.ca}

\dedicatory{}
\subjclass[2010]{}
\begin{abstract}
We study the motion of smooth, closed, strictly convex hypersurfaces in $\mathbb{R}^{n+1}$ expanding in the direction of their normal vector field with speed depending on the $k$th elementary symmetric polynomial of the principal radii of curvature $\sigma_k$ and support function $h$. A homothetic self-similar solution to the flow that we will consider in this paper, if exists, is a solution of the well-known $L_p$-Christoffel-Minkowski problem $\varphi h^{1-p}\sigma_k=c$. Here $\varphi$ is a preassigned positive smooth function defined on the unit sphere, and $c$ is a positive constant. For $1\leq k\leq n-1, p\geq k+1$, assuming the spherical hessian of $\varphi^{\frac{1}{p+k-1}}$ is positive definite, we prove the $C^{\infty}$ convergence of the normalized flow to a homothetic self-similar solution. One of the highlights of our arguments is that we do not need the constant rank theorem/deformation lemma of \cite{Guan-Ma} and thus we give a partial answer to a question raised in \cite{GuanXia}. Moreover, for $k=n, p\geq n+1$, we prove the $C^{\infty}$ convergence of the normalized flow to a homothetic self-similar solution without imposing any further condition on $\varphi.$ In the final section of the paper, for $1\leq k<n$, we will give an example that spherical hessian of $\varphi^{\frac{1}{p+k-1}}$ is negative definite at some point and the solution to the flow loses its smoothness.
\end{abstract}

\maketitle
\section{An expanding flow}
Suppose $F_0:M^n\to \mathbb{R}^{n+1}$ is a smooth parametrization of a closed, strictly convex hypersurface $M_0$ and suppose the origin of $\mathbb{R}^{n+1}$ is in the interior of the region enclosed by $M_0$. In this paper, we study the long-time behavior of a family of hypersurfaces ${M_t}$ given by the smooth map $F:M^n\times [0,T)\to \mathbb{R}^{n+1}$ satisfying the initial value problem
\begin{align}\label{F-param}
\left\{
  \begin{array}{ll}
  \partial_tF(x,t)=\varphi(\nu(x,t))\langle F(x,t),\nu(x,t)\rangle^{2-p}\frac{E_{n-k}}{E_n}(x,t)\nu(x,t); \\
   F(\cdot,0)=F_0(\cdot),
  \end{array}
\right.
\end{align}
where $p\in \mathbb{R},$ $E_i$ is the $i$th elementary symmetric polynomial of principal curvatures for $0\leq i\leq n$ normalized so that $E_{i}(1,\ldots,1)=1$ (and $E_0\equiv 1$), $\nu(\cdot,t)$ is the outer unit normal vector of $M_t:=F(M^n,t)$ and $\varphi$ is a positive smooth function defined on the unit sphere $\mathbb{S}^n.$

Assuming $M_t$ is strictly convex, its support function as a function on the unit sphere is given by
\begin{align*}
h_{M_t}(x)=h(x,t)&:=\langle F(\nu^{-1}(x,t),t),x\rangle.
\end{align*}
Write $\bar{g}$ and $\bar{\nabla}$ for the standard round metric and the Levi-Civita connection of $\mathbb{S}^n$. Recall that the principal radii of curvature are the eigenvalues of the matrix
\[r_{ij}:=\bar{\nabla}_i\bar{\nabla}_j h+\bar{g}_{ij}h\]
with respect to $\bar{g}$.
Also write $\sigma_k$ for the $k$th elementary symmetric polynomial of the principal radii of curvature, normalized so that $\sigma_k(1,\ldots,1)=1.$

We also define
\begin{align}\label{norsol}
\tilde{M}_{t}:=\left\{
  \begin{array}{ll}
    \left(\frac{\int_{\mathbb{S}^n}\frac{1}{\varphi}dx}{\int_{\mathbb{S}^n}\frac{h^p(x,t)}{\varphi}dx}\right)^{\frac{1}{p}}M_t, & \hbox{if $p\neq 0$;} \\
    \exp\left(-\frac{1}{\int_{\mathbb{S}^n}\frac{1}{\varphi}dx}\int_{\mathbb{S}^n}\frac{\log h(x,t)}{\varphi}dx\right)M_t, & \hbox{if $p= 0$.}
  \end{array}
\right.
\end{align}
Here $dx$ is the Lebesgue measure on $\mathbb{S}^n$ and $\omega_n=\int_{\mathbb{S}^n}dx.$

By direct calculation, we find $h:\mathbb{S}^n\times [0,T)\to \mathbb{R}$ satisfies
\begin{align}\label{unf}
\left\{
  \begin{array}{ll}
\partial_th=\varphi h^{2-p}\sigma_k;\\
h(x,0)=\langle F_0(\nu^{-1}(x)),x\rangle.
  \end{array}
\right.
\end{align}
We consider a normalization of the flow (\ref{unf}) given by
\begin{align}\label{nf}
\partial_{\tau}h=\varphi h^{2-p}\sigma_k-h\frac{\int_{\mathbb{S}^n}h\sigma_kdx}{\int_{\mathbb{S}^n}\frac{1}{\varphi}dx}.
\end{align}
\begin{Cav}
We always distinguish between the solutions to (\ref{unf}) and (\ref{nf}) respectively through the parameters $t,\tau$.
\end{Cav}
Note that for $p\neq0,$
\begin{align*}
\frac{d}{d\tau}\int_{\mathbb{S}^n}\frac{h^p(x,\tau)}{\varphi(x)}dx&=p\int_{\mathbb{S}^n}h\sigma_kdx-
p\int_{\mathbb{S}^n}\frac{h^p}{\varphi}dx\frac{\int_{\mathbb{S}^n}h\sigma_kdx}{\int_{\mathbb{S}^n}\frac{1}{\varphi}dx}\\
&=p\frac{\int_{\mathbb{S}^n}h\sigma_kdx}{\int_{\mathbb{S}^n}\frac{1}{\varphi}dx}\left(\int_{\mathbb{S}^n}\frac{1}{\varphi}dx-\int_{\mathbb{S}^n}\frac{h^p}{\varphi}dx\right).
\end{align*}
If the solution to (\ref{nf}) at time $\tau=0$ satisfies $\int_{\mathbb{S}^n}\frac{h^p}{\varphi}dx=\int_{\mathbb{S}^n}\frac{1}{\varphi}dx$, then at any later time this identity still holds. If $p=0,$ we always have $\frac{d}{d\tau}\int_{\mathbb{S}^n}\frac{\log h(x,\tau)}{\varphi(x)}dx=0.$ Note the support functions of $\tilde{M}_t$ after a suitable time re-parametrization solve (\ref{nf}).

Our motivation to study the flow (\ref{unf}) is due to the significance of its solitons in convex geometry. A positive homothetic self-similar solution of (\ref{unf}), when exists, is a solution to
\begin{align}\label{soliton}
\varphi h^{1-p}\sigma_k=c.
\end{align}
for some $c>0.$ One would like to find necessary and sufficient conditions on a function $\varphi$ such that a positive strictly convex solution exists. Here the strict convexity of a solution, $h$, is understood as the strict convexity of the associated closed hypersurface. The pairs $(p=1,k=1)$, $(p=1,k=n)$, $(p\neq 1,k=n)$ of this equation are known in order as the Christoffel problem, the Minkowski problem and the $L_p$-Minkowski problem. In general, this equation is known as the $L_p$-Christoffel-Minkowski problem. This equation is of considerable interest in convex geometry, and it is related to the problem of existence of a convex body (a compact convex set with non-empty interior) whose $L_p$ surface area of order $n-k$ is prescribed.

Let us briefly explain how (\ref{soliton}) arises naturally in the $L_p$ Brunn-Minkowski theory. Good references for this material are \cites{Lut1,Schneider}. Let $p\geq 1$ and $\zeta,\eta\geq 0$ and let $K,L$ be two convex bodies with the origin of $\mathbb{R}^{n+1}$ in their interiors. In the following, $\zeta\cdot K:=\zeta^{\frac{1}{p}}K$ and $\eta\cdot L:=\eta^{\frac{1}{p}}L$. Define the $L_p$-linear combination $\zeta\cdot K+_{p}\eta\cdot L$ as the convex body whose support function is given by $(\zeta h_K^p+\eta h_L^p)^{\frac{1}{p}}.$ The mixed $L_p$-Quermassintegrals $W_{p,0}(K,L),\ldots, W_{p,n}(K,L)$ are defined as the first variation of the usual Quermassintegrals\footnote{For a convex body $K$, $W_0(K),\ldots, W_{n+1}(K)$ notate the Quermassintegrals of $K$. In particular, $W_0(K)$ is volume of $K$, $nW_1(K)$ is the surface area of $K$ and $W_{n+1}(K)$ is the volume of the unit ball.} with respect to $L_p$-sum:
\[\frac{n+1-k}{p}W_{p,k}(K,L)=\lim_{\varepsilon\to 0^+}\frac{W_k(K+_p\varepsilon\cdot L)-W_k(K)}{\varepsilon}.\]
Aleksandrov, Fenchel and Jessen for $p=1$ and Lutwak \cite{Lut1} for $p>1$ have shown that for each $k=0,\ldots,n$, there exists a Borel measure $S_{p,k}(K,\cdot)$ on $\mathbb{S}^n$, $L_p$ surface area measure of order $k$, such that
\[W_{p,k}(K,L)=\frac{1}{n+1}\int_{\mathbb{S}^n}h_L^p(u)dS_{p,k}(K,u).\]
Moreover, $S_{p,k}(K,\cdot)$ is absolutely continuous with respect to $k$th surface area measure of $K$, $S_k(K,\cdot)$, and has the Radon-Nikodym derivative
\[\frac{dS_{p,k}(K,\cdot)}{dS_k(K,\cdot)}=h_K^{1-p}(\cdot).\]
In addition, if the boundary of $K$ is a $C^2$-smooth hypersurface with everywhere positive principal curvatures, then
\[dS_{p,k}(K,\cdot)=h_K^{1-p}(\cdot)\sigma_{n-k}(K,\cdot)dx.\]

If $p=1,$ a necessary condition for the existence of a solution to (\ref{soliton}) is that $\varphi$ must satisfy the vector equation
\begin{align}\label{closing}
\int_{\mathbb{S}^n}\frac{x}{\varphi(x)}dx=0.
\end{align}
Miraculously this condition suffices for the Minkowski problem; see, for example, \cite{Chengyau}. The $L_p$-Minkowski problem is also well-understood (except the case $p\leq -n-1$) and we refer the reader to the essential papers \cites{Lut1,Lut2,Lut3,CBC,chwang} for motivation and the most comprehensive list of results, see also \cite{Schneider}*{Chapter 9.2, Notes for Section 9.2}. An application of the existence of solutions to the $L_p$-Minkowski problem appears in Lutwak, Yang, Zhang \cite{LYZ}. If $p=1, k< n,$ much less is known and in addition to (\ref{closing}) further restrictions need to be imposed on $\varphi$. For example, let us consider the case when $\varphi$ is rotationally symmetric. A function $\varphi$ defined on the unit sphere is said to be rotationally symmetric if $\varphi(\theta)=\varphi(x_1,\ldots,x_{n+1})$ with $x_{n+1}=\sin(\theta)$ where $\theta\in [-\frac{\pi}{2},\frac{\pi}{2}]$. Note that $\theta$ is the angle that the vector from the origin to $(x_1,\ldots,x_{n+1})$ makes with $x_{n+1}=0.$ In \cite{Fi3}, Firey has found that in order for a continuous function $1/\varphi $ to be the $k$th elementary symmetric function of the principal radii of a $C^2$ smooth, closed strictly convex hypersurface of revolution, it is necessary and sufficient that in some coordinates on $\mathbb{S}^n$, $\varphi$ is a function of the latitude $\theta$ alone, and over $-\frac{\pi}{2}<\theta<\frac{\pi}{2}:$
\begin{description}
  \item[i] $\frac{1}{\varphi}$ is continuous and has finite limits as $\theta$ tends to $\pm\frac{\pi}{2}$,
  \item[ii] $\int_{\theta}^{\frac{\pi}{2}}\frac{\cos^{n-1}(\alpha)\sin(\alpha)}{\varphi(\alpha)}d\alpha>0$ and zero for $\theta=-\frac{\pi}{2},$
  \item[iii] $\frac{1}{\varphi(\theta)}>\frac{n-k}{\cos^n(\theta)}\int_{\theta}^{\frac{\pi}{2}}\frac{\cos^{n-1}(\alpha)\sin(\alpha)}{\varphi(\alpha)}d\alpha.$
\end{description}
Due to symmetry, the assumption for $\theta=-\frac{\pi}{2}$ in item (ii) is the same as the closure equation (\ref{closing}). The main consequence of item (iii) is that the principal radii of curvature are positive:
\begin{align*}
&\frac{1}{\varphi(\theta)}-\frac{n-k}{\cos^n(\theta)}\int_{\theta}^{\frac{\pi}{2}}\frac{\cos^{n-1}(\alpha)\sin(\alpha)}{\varphi(\alpha)}d\alpha\\
&=\binom{n-1}{k-1}(h''(\theta)+h(\theta))(h(\theta)-h'(\theta)\tan\theta)^{k-1}.
\end{align*}
In \cite{pog}, Pogorelov proved if $\varphi^{-1}-(\varphi^{-1})_{ss}> 0$ on every great circle parameterized by arc-length $s$, then $1/\varphi$ is the sum of the principal radii of curvature in Euclidean 3-space (this is not a necessary condition). The case $p=1,k=1$ without any dimensional restriction was eventually solved by Firey \cites{Fi1,Fi2} where he gave a necessary and sufficient condition, settling a hundred year old problem posed by Christoffel.\footnote{Firey also explains in \cite{Fi1}*{page 11} how Pogorelov's condition connects to his.} An application of Firey's \cite{Fi1} existence result to the study of surfaces of constant
width appears in Fillmore \cite{Fillmore}. The solution to Christoffel's problem was independently discovered by Berg \cite{Berg}. See also \cite{Schneider}*{Chapter 8.3.2} for the explicit construction of the solution to the Christoffel problem and \cite{GYY}*{Theorem 6.1} for the corresponding regularity properties and \cite{Schneider}*{Notes for Section 8.4}. In \cite{Guan-Ma}, Guan-Ma proved a deformation lemma which allowed them to establish if a function $\varphi\in C^2(\mathbb{S}^{n})$ is $k$-convex, e.q., $\bar{\nabla}_i\bar{\nabla}_j\varphi^{\frac{1}{k}}+\bar{g}_{ij}\varphi^{\frac{1}{k}}$ is non-negative definite, then the equation (\ref{soliton}) for $p=1,k<n$ has a strictly convex solution. Note that Guan-Ma's condition for $p=1,k=1$ is weaker than Pogorelov's condition. Later in \cite{HMS}, using the deformation lemma, Hu-Ma-Shen proved that if $p\geq k+1,k<n$ and $\varphi\in C^2(\mathbb{S}^{n})$ is $(p+k-1)$-convex, then (\ref{soliton}) admits a positive strictly convex solution.\footnote{The statement of Hu-Ma-Shen's theorem is erroneous and in item (i) it should be read ``if $f^{-\frac{1}{p+k-1}}$ is spherical convex" and in item (ii) it should be read ``if $f^{-\frac{1}{2k}}$ is spherical convex".} Recently, for $1<p<k+1$ and for even prescribed data, under the $(p+k-1)$-convexity of $\varphi,$ an existence result was proved by Guan and Xia in \cite{GuanXia} using a refined gradient estimate and the constant rank theorem.

Before we state our main theorems, we draw attention to an interesting feature of the flow (\ref{unf}); however, this property is not used in this paper. Suppose that for a positive, smooth rotationally symmetric $\varphi$ and a smooth, rotationally symmetric, strictly convex hypersurface $M_0$ with the support function $h_0$ we have
\begin{description}
  \item[ii] $\int_{\theta}^{\frac{\pi}{2}}\frac{\cos^{n-1}(\alpha)\sin(\alpha)h_0^{p-1}(\alpha)}{\varphi(\alpha)}d\alpha>0$ and zero for $\theta=-\frac{\pi}{2},$
  \item[iii] $\frac{h_0^{p-1}(\theta)}{\varphi(\theta)}>\frac{n-k}{\cos^n(\theta)}
      \int_{\theta}^{\frac{\pi}{2}}\frac{\cos^{n-1}(\alpha)\sin(\alpha)h_0^{p-1}(\alpha)}{\varphi(\alpha)}d\alpha.$
\end{description}
If we start the flow (\ref{unf}) from $M_0$, then for all $t>0,$ $M_t$ satisfies the previous two properties provided $p>1.$ To see this for the item (iii), note that
\begin{align*}
\frac{d}{dt}&\left(\frac{h^{p-1}(\theta,t)}{\varphi(\theta)}-\frac{n-k}{\cos^n(\theta)}
      \int_{\theta}^{\frac{\pi}{2}}\frac{\cos^{n-1}(\alpha)\sin(\alpha)h^{p-1}(\alpha,t)}{\varphi(\alpha)}d\alpha\right)=\\
      &(p-1)\left(\sigma_k(\theta,t)-\frac{n-k}{\cos^n(\theta)}
      \int_{\theta}^{\frac{\pi}{2}}\cos^{n-1}(\alpha)\sin(\alpha)\sigma_k(\alpha,t)d\alpha\right)>0.
\end{align*}
One can see similarly that item (ii) is preserved along the flow. For the case $\varphi\equiv1,k=n, p=-n-1$, preserving a property similar to (ii) played a role in the proofs of \cite{ivakifunc}.

In this paper, we prove the following theorems about the asymptotic behavior of the flow.
\begin{theorem}\label{main theorem1}
Suppose $p\geq k+1, k<n$ and $\varphi\in C^{\infty}(\mathbb{S}^n)$ is a positive function such that $\bar{\nabla}_i\bar{\nabla}_j\varphi^{\frac{1}{p+k-1}}+\bar{g}_{ij}\varphi^{\frac{1}{p+k-1}}$ is positive definite. Then there exists a unique smooth, closed strictly convex solution $\{M_t\}$ to (\ref{F-param}) such that $\{\tilde{M}_t\}$ converges in $C^{\infty}$ to a smooth, closed strictly convex solution hypersurface whose support function is positive and solves (\ref{soliton}).
\end{theorem}
Our proof of the convergence to solitons does not employ the deformation lemma (constant rank theorem) and thus provides a partial answer to the following question raised in \cite{GuanXia}: Is there a direct effective way to derive an estimate for $r_{ij}$ from below under the same convexity conditions
without using the constant rank theorem? Here our parabolic approach to the problem (\ref{soliton}) allows us to obtain a uniform positive lower bound on $r_{ij}$ along the normalized flow by using a very simple auxiliary function (see Lemma \ref{lower and upper princpal bounds} below) and hence we can avoid the constant rank theorem when the assumptions of Theorem \ref{main theorem1} are satisfied.

In the last section, for $1\leq k<n, p+k-1>0$ we show the existence of a rotationally symmetric $\varphi$ with $((\varphi^{\frac{1}{p+k-1}})_{\theta\theta}+\varphi^{\frac{1}{p+k-1}})\big|_{\theta=0}<0$ and a smooth, closed, strictly convex initial hypersurface for which the solution to the flow (\ref{F-param}) with $k<n$ will lose smoothness. Therefore $(p+k-1)$-convexity of $\varphi$ is essential to ensure the smoothness of the solution is preserved.

For $k=n, p\geq n+1$, we can improve \cite{BIS}*{Theorem 1} by dropping the evenness assumption and allowing general $\varphi$.
\begin{theorem}\label{main theorem1a}
Suppose $k=n, p\geq n+1$ and $\varphi\in C^{\infty}(\mathbb{S}^n)$ is a positive function. Then there exists a unique smooth, closed strictly convex solution $\{M_t\}$ to (\ref{F-param}) such that $\{\tilde{M}_t\}$ converges in $C^{\infty}$ to a smooth, closed strictly convex solution hypersurface whose support function is positive and solves (\ref{soliton}).
\end{theorem}
We should point out the only new ingredient required to prove this last theorem is the gradient estimate established in Lemma \ref{grad estimate}; this allows us to obtain uniform lower and upper bounds on the support function of the normalized solution even if the initial hypersurface is not origin-symmetric. In particular, the curvature estimate of \cite{BIS}*{Lemma 8} is crucial.

Finally in view of Chow-Gulliver's gradient estimate \cite{CHGU} for the case $p>2,\varphi\equiv 1$, we have the following result.
\begin{theorem}\label{main theorem2}
Suppose $p>2$ and $\varphi\equiv 1.$ Then there exists a unique smooth, closed strictly convex solution $\{M_t\}$ to (\ref{F-param}) such that $\{\tilde{M}_t\}$ converges in $C^{\infty}$ to the unit sphere.
\end{theorem}
To conclude this section, we draw attention to some earlier works on the flow (\ref{unf}).
For $p=2,\varphi\equiv 1$, the $C^1$ convergence was established by Chow-Tsai \cite{BT} and recently the $C^{\infty}$ convergence was proved by Gerhardt in \cite{Ger}. For $p>2,~\varphi\equiv 1$, the $C^1$ convergence follows from the work of Chow-Gulliver \cite{CHGU} (up-to showing convexity is preserved). For $p=-n-1,k=n,\varphi\equiv1$, the flow was studied in \cites{ivakitrans,ivakifunc} and for $p>-n-1,k=n, \varphi\not\equiv 1$ in \cite{BIS}.
\section*{Acknowledgment}
M.I. has been supported in part by a Jerrold E. Marsden Fellowship. I would like to thank the Fields Institute for providing an excellent research environment during the thematic program on Geometric Analysis.
\section{Regularity estimates}\label{reg est}
For convenience we put
\[\eta:=\frac{\int_{\mathbb{S}^n}h\sigma_kdx}{\int_{\mathbb{S}^n}\frac{1}{\varphi}dx},\quad \Theta:=\varphi h^{2-p},\quad \mathcal{L}:=\Theta \sigma_k^{ab}\bar{\nabla}_a\bar{\nabla}_b,\quad \rho:=\sqrt{h^2+|\bar{\nabla}h|^2}.\]
If $h\in C^{\infty}(\mathbb{S}^{n})$ determines a smooth, closed strictly convex hypersurface, we write $[h]$ for the associated hypersurface. For such a hypersurface define
\begin{align*}
\mathcal{A}_{k,p}^{\varphi}[h]:=\left\{
  \begin{array}{ll}
    \int_{\mathbb{S}^n}h\sigma_kdx\left(\int_{\mathbb{S}^n}\frac{h^p}{\varphi}dx\right)^{-\frac{k+1}{p}}, & \hbox{if $p\neq 0$;} \\
    \exp\left(-\frac{k+1}{\int_{\mathbb{S}^n}\frac{1}{\varphi}dx}\int_{\mathbb{S}^n}\frac{\log h}{\varphi}dx\right)\int_{\mathbb{S}^n}h\sigma_kdx, & \hbox{if $p=0$.}
  \end{array}
\right.
\end{align*}
The functionals $\mathcal{A}_{k,p}^{\varphi}$ are well-known and have appeared for example in \cite{BA1}.
\begin{lemma}\label{monotonicity}
$\mathcal{A}_{k,p}^{\varphi}[h(\cdot,\tau)]$ is non-decreasing.
\end{lemma}
\begin{proof}
We only consider the case $p\neq 0.$ Using the divergence theorem, we calculate
\begin{align*}
\frac{\frac{d}{d\tau}\mathcal{A}_{k,p}^{\varphi}[h(\cdot,\tau)]}{(k+1)\left(\int_{\mathbb{S}^n}\frac{h^p}{\varphi}dx\right)^{-\frac{k+1}{p}-1}}&=\int_{\mathbb{S}^n}\varphi h^{2-p}\sigma_k^2dx\int_{\mathbb{S}^n}\frac{h^p}{\varphi}dx-\left(\int_{\mathbb{S}^n}h\sigma_kdx\right)^2.
\end{align*}
Therefore by the H\"{o}lder inequality $\mathcal{A}_{k,p}^{\varphi}[h(\cdot,\tau)]$ is non-decreasing along the flow.
\end{proof}
\begin{lemma}\label{eta} Suppose $p\geq 2.$
$\eta(\tau)$ is uniformly bounded above and below.
\end{lemma}
\begin{proof}
The uniform lower bound on $\eta(\tau)$ follows from Lemma \ref{monotonicity}.
To prove that $\eta$ is uniformly bounded above we proceed as follows. Along the flow $\partial_th=\sigma_k$, $\mathcal{A}_{k,2}^{1}[h(\cdot,t)]$ is monotone. Thus by the result of \cite{Ger}, for any  smooth, closed strictly convex hypersurface with support function $h$ we have
$\mathcal{A}_{k,2}^{1}[h]\leq \mathcal{A}_{k,2}^{1}[1].$
Now observe that for $p\geq 2,$ the inequality
\[\left(\int_{\mathbb{S}^n}h^2dx\right)^{-\frac{k+1}{2}}\geq \left(\left(\int_{\mathbb{S}^n}h^pdx\right)^{\frac{2}{p}}\omega_n^{1-\frac{2}{p}}\right)^{-\frac{k+1}{2}}\]
gives
\[\int_{\mathbb{S}^n}h\sigma_kdx\left(\left(\int_{\mathbb{S}^n}h^pdx\right)^{\frac{2}{p}}\omega_n^{1-\frac{2}{p}}\right)^{-\frac{k+1}{2}}\leq \mathcal{A}_{k,2}^{1}[1].\]
So $\mathcal{A}_{k,2}^{\varphi}$ is bounded above. The proof of the lemma is done.\footnote{In fact, only knowing the asymptotic behavior of the flow $\partial_th=\sigma_1$ is sufficient here; using quermassintegral inequalities we can control $\int_{\mathbb{S}^n}h\sigma_kdx$ from above by $(\int_{\mathbb{S}^n}h\sigma_1dx)^{\frac{k+1}{2}}$.}
\end{proof}
In the next lemma $r^{ij}$ signifies the entries of the inverse matrix of $[r_{ij}].$
\begin{lemma}\label{ev equ}
The following evolution equation holds along the flow (\ref{nf}).
\begin{align*}
(\partial_{\tau}-\mathcal{L})r_{ij}=&\Theta \sigma_k^{ab,mn}\bar{\nabla}_ir_{ab}\bar{\nabla}_jr_{mn}+(k+1)\Theta\sigma_k\bar{g}_{ij}\\
&-\Theta\sigma_k^{ab}\bar{g}_{ab}r_{ij}
+\Theta(\sigma_k^{ai}r_{aj}-\sigma_k^{aj}r_{ai})\\
&+\bar{\nabla}_i\Theta\bar{\nabla}_j\sigma_k+\bar{\nabla}_j\Theta\bar{\nabla}_i\sigma_k
+\sigma_k\bar{\nabla}_i\bar{\nabla}_j\Theta-\eta r_{ij},\\
\left(\partial_\tau-\mathcal{L}\right) r^{ij}=&-(k+1)\Theta\sigma_kr^{ip}r^{jp}+\Theta\sigma_k^{ab}\bar{g}_{ab}r^{ij}\\
&-\Theta r^{il}r^{js}(2\sigma_k^{am}r^{nb}+\sigma_k^{ab,mn})\bar{\nabla}_lr_{ab}\bar{\nabla}_sr_{mn}\\
&-\Theta r^{ip}r^{jq}(\sigma_k^{ap}r_{aq}-\sigma_k^{aq}r_{ap})\\
&-r^{ia}r^{jb}(\bar{\nabla}_a\Theta\bar{\nabla}_b\sigma_k+\bar{\nabla}_b\Theta\bar{\nabla}_a\sigma_k+\sigma_k\bar{\nabla}_a\bar{\nabla}_b\Theta)+\eta r^{ij},\\
\left(\partial_\tau-\mathcal{L}\right)h=&(1-k)\Theta \sigma_k+\Theta h\sigma_k^{ij}\bar{g}_{ij}-\eta h,\\
\left(\partial_\tau-\mathcal{L}\right) (\varphi h^{2-p}\sigma_k)=&(2-p)\varphi^2h^{3-2p}\sigma_k^2+\varphi^2 h^{4-2p}\sigma_k\sigma_k^{ij}\bar{g}_{ij}-(2-p+k)\eta\varphi h^{2-p}\sigma_k,\\
\left(\partial_\tau-\mathcal{L}\right)\frac{\rho^2}{2}=&(k+1)h\Theta\sigma_k-\eta \rho^2+\sigma_k\bar{g}^{ij}\bar{\nabla}_ih\bar{\nabla}_j\Theta-\Theta\sigma_k^{ab}r_a^mr_{mb}.
\end{align*}
\end{lemma}
\begin{proof}
For the computation of the evolution equations of $r_{ij}$ and $r^{ij}$ see \cites{BT,JU}. Deriving the evolution equations of $h,\varphi h^{2-p}\sigma_k$ is straightforward.
For $\rho^2/2$ we calculate
\begin{align*}
\left(\partial_\tau-\mathcal{L}\right)\frac{\rho^2}{2}
=&h\partial_\tau h+\bar{g}^{ij}\bar{\nabla}_ih\bar{\nabla}_j\partial_\tau h\\
&-\Theta\sigma_k^{ab}\left(h\bar{\nabla}_a\bar{\nabla}_bh+\bar{\nabla}_ah\bar{\nabla}_bh+\bar{\nabla}_mh\bar{\nabla}_a\bar{\nabla}_b\bar{\nabla}_mh+
\bar{\nabla}_a\bar{\nabla}_mh\bar{\nabla}_b\bar{\nabla}_mh
\right)\\
=&h\partial_\tau h-\eta|\bar{\nabla}h|^2+\sigma_k\bar{g}^{ij}\bar{\nabla}_ih\bar{\nabla}_j\Theta
+\Theta\bar{g}^{ij}\bar{\nabla}_ih\bar{\nabla}_j\sigma_k\\
&-\Theta\sigma_k^{ab}\left(\bar{\nabla}_ah\bar{\nabla}_bh+\bar{\nabla}_mh\bar{\nabla}_a(r_{bm}-\bar{g}_{bm}h)
\right)\\
&-\Theta\sigma_k^{ab}(r_{am}-\bar{g}_{am}h)(r_{bm}-\bar{g}_{bm}h)\\
&-\Theta\sigma_k^{ab}h(r_{ab}-\bar{g}_{ab}h)\\
=&h\partial_\tau h-\eta|\bar{\nabla}h|^2+\sigma_k\bar{g}^{ij}\bar{\nabla}_ih\bar{\nabla}_j\Theta+kh\Theta\sigma_k-\Theta\sigma_k^{ab}r_a^mr_{mb}\\
=&(k+1)h\Theta\sigma_k-\eta \rho^2+\sigma_k\bar{g}^{ij}\bar{\nabla}_ih\bar{\nabla}_j\Theta-\Theta\sigma_k^{ab}r_a^mr_{mb}.
\end{align*}
\end{proof}
\begin{lemma}\label{lower and upper speed bounds}
Suppose $p\geq k+1.$ Then
$\varphi h^{1-p}\sigma_k(\cdot,\tau)$ remains uniformly bounded above and below.
\end{lemma}
\begin{proof}
By Lemma \ref{ev equ} we have
\begin{align}\label{ev of speed divided by h}
\partial_{\tau} \left(\varphi h^{1-p}\sigma_k\right)
=&\varphi h^{2-p}\sigma_k^{ij}\bar{\nabla}_i\bar{\nabla}_j\left(\varphi h^{1-p}\sigma_k\right)+2\varphi h^{1-p}\sigma_k^{ij}\bar{\nabla}_i(\varphi h^{1-p}\sigma_k)\bar{\nabla}_jh\nonumber\\
&+(1+k-p)\left(\varphi h^{1-p}\sigma_k\right)^2+(p-k-1)\eta\varphi h^{1-p}\sigma_k.
\end{align}
Since we have control over $\eta$, the claim follows from the maximum principle.
\end{proof}
In the next lemma, using the concavity of $\sigma_k^{\frac{1}{k}}$ we obtain a gradient estimate for $\log h(\cdot,\tau)$ provided $p\geq k+1.$ In general, due to the examples in \cite{GuanXia}, such an estimate does not exist for $1<p<k+1$.
\begin{lemma}\label{grad estimate}
Suppose $p\geq k+1$. There exists a positive constant $\gamma$ depending only on the initial hypersurface and $\varphi$ such that
$|\bar{\nabla}\log h(\cdot,\tau)|\le \gamma.$
\end{lemma}
\begin{proof}
The proof is a parabolic version of the estimates in \cite{HMS}*{Lemma 2}. Using Lemma \ref{ev equ}
we deduce
\begin{align*}
(\partial_{\tau}-\mathcal{L})(\rho^2-Ah^2)=&2(k+1)h\Theta\sigma_k-2\eta\rho^2+2\sigma_k\bar{g}^{ij}\bar{\nabla}_ih\bar{\nabla}_j\Theta-2\Theta\sigma_k^{ab}r_a^mr_{mb}\\
&-2A h((1-k)\Theta \sigma_k+\Theta h\sigma_k^{ij}\bar{g}_{ij}-\eta h)+2A\Theta\sigma_k^{ab}\bar{\nabla}_ah\bar{\nabla}_bh.
\end{align*}
Let us put $v:=\log h$. Pick $A>1$ such that
$$(\rho^2-Ah^2)(\cdot,0)<0.$$ We will show that this inequality will be preserved, perhaps for a larger value of $A$ to be determined later. If it were otherwise, there would be a point
 $(u_{\tau},\tau)$ with $\tau>0$ that for the first time $(\rho^2-Ah^2)(u_{\tau},\tau)=0$. At this point
$\bar{\nabla}|\bar{\nabla}v|^2=0$ and we may choose an orthonormal frame $\{e_i\}$ such that
\[\bar{\nabla}v=(\bar{\nabla}_{e_1}v)e_1.\]
For the rest of the proof it is more convenient to put
\[v_i:=\bar{\nabla}_{e_i}v,\quad v_{ij}:=\bar{\nabla}_i\bar{\nabla}_jv.\]
Since $\bar{\nabla}|\bar{\nabla}v|^2=0$ at $u_{\tau}$, a rotation of $\{e_i\}_{i\geq 2}$ diagonalizes $(v_{ij})$ at $u_{\tau},$
\[a_{ij}:=v_{ij}+v_iv_j+\delta_{ij}=\operatorname{diag}(1+v_1^2,1+v_{22},\ldots,1+v_{nn}).\]
Also, note that
\[\sigma_k(h_{ij}+\delta_{ij}h)=h^k\sigma_k(a_{ij}),\quad \sigma_k^{cd}(h_{ij}+\delta_{ij}h)=h^{k-1}\sigma_k^{cd}(a_{ij}).\]
Therefore at $(u_{\tau},\tau)$ we have
\begin{align*}0\leq& (k+1)+(2-p)v_1^2+v_1(\log\varphi)_1-\frac{\sigma_k^{ll}(a_{ij})}{\sigma_k(a_{ij})}(a_{ll})^2\\
&+A(k-1)-A\sum_l\frac{\sigma_k^{ll}(a_{ij})}{\sigma_k(a_{ij})}+A\frac{\sigma_k^{11}(a_{ij})}{\sigma_k(a_{ij})}v_1^2.
\end{align*}
Since $A=1+v_1^2=a_{11},$ we may rewrite the previous estimate as
\begin{align*}
0\leq& 2k+(1+k-p)(A-1)+v_1(\log\varphi)_1-\frac{\sigma_k^{11}(a_{ij})}{\sigma_k(a_{ij})}A^2\\
&-A\sum_l\frac{\sigma_k^{ll}(a_{ij})}{\sigma_k(a_{ij})}+A(A-1)\frac{\sigma_k^{11}(a_{ij})}{\sigma_k(a_{ij})}.
\end{align*}
Thus for $p>k+1$ we arrive at
\[(p-k-1)(A-1)\leq 2k+|(\log\varphi)_1|\sqrt{A-1}.\]
Choosing $A$ large enough ensures that $\rho^2-Ah^2$ always remains negative.

If $p=k+1$, by Lemma \ref{lower and upper speed bounds}, $\sigma_k(a_{ij})$ is uniformly bounded above. Also, since $\sigma_k^{\frac{1}{k}}$ is concave, we have $\sum_l\sigma_{k}^{ll}(a_{ij})\geq k\sigma_k^{\frac{k-1}{k}}(a_{ij})$; see, for instance, \cite{AMZ}. Thus for a positive constant $c_1$ depending only on the initial hypersurface and $\varphi$ we have
\begin{align*}
c_1A\leq A\sum_l\frac{\sigma_k^{ll}(a_{ij})}{\sigma_k(a_{ij})}\leq& 2k+|(\log\varphi)_1|\sqrt{A-1}.
\end{align*}
Choosing $A$ large enough proves the claim.
\end{proof}
\begin{lemma}\label{key lemma}
Suppose $p\geq k+1.$ There exist positive constants $a,b,c,d$ depending only on the initial hypersurface and $\varphi$ such that
$a\leq h(\cdot,\tau)\leq b$ and
$c\leq \sigma_k(\cdot,\tau)\leq d.$
\end{lemma}
\begin{proof}
Along the normalized flow $\int_{\mathbb{S}^n}\frac{h^p(x,\tau)}{\varphi(x)}dx$ is constant and
\[\frac{h_{\min}^p(\tau)}{\varphi_{\max}}\leq \frac{1}{\omega_n}\int_{\mathbb{S}^n}\frac{h^p(x,\tau)}{\varphi(x)}dx\leq \frac{h_{\max}^p(\tau)}{\varphi_{\min}}.\]
Therefore Lemma \ref{grad estimate} gives uniform lower and upper bounds on $h(\cdot,\tau).$ Now the lower and upper bounds on $\sigma_k(\cdot,\tau)$ follow from Lemma \ref{lower and upper speed bounds}.
\end{proof}
Next we will obtain a lower bound on the principal radii of curvature under an additional assumption on $\varphi$. This in turns implies that the normalized hypersurfaces are uniformly convex. It is only in the following lemma that we require $\bar{\nabla}_i\bar{\nabla}_j\varphi^{\frac{1}{p+k-1}}+\bar{g}_{ij}\varphi^{\frac{1}{p+k-1}}$ to be positive definite. Our example in the final section shows that one cannot hope for a positive lower for $r_{ij}$ if $\bar{\nabla}_i\bar{\nabla}_j\varphi^{\frac{1}{p+k-1}}+\bar{g}_{ij}\varphi^{\frac{1}{p+k-1}}$ is negative definite at some point.
\begin{lemma}\label{lower and upper princpal bounds}
Let $p\geq k+1$. Suppose $\bar{\nabla}_i\bar{\nabla}_j\varphi^{\frac{1}{p+k-1}}+\bar{g}_{ij}\varphi^{\frac{1}{p+k-1}}$ is positive definite.
Then the principal curvatures satisfy
$l\leq \kappa_i(\cdot,\tau)\leq L$
for some positive constants $l,L$ depending only on the initial hypersurface and $\varphi.$
\end{lemma}
\begin{proof} We provide two proofs.

\emph{First proof}.
We apply the maximum principle to $\frac{r^{ij}}{h}$; see also \cite{JU1}*{Lemma 3.3}.
By a rotation of frame we may assume the maximum eigenvalue of $\frac{r^{ij}}{h}$ over $\mathbb{S}^n$ at time $\tau$ is attained at a point $u_{\tau}$ in the direction of unit tangent vector $e_1\in T_{u_{\tau}}\mathbb{S}^n$. In particular, $r^{ij}=0$ for $i\neq j.$ Using Lemma \ref{ev equ} we calculate
\begin{align*}
\partial_\tau \frac{r^{11}}{h}
=&\Theta\sigma_k^{ab}\bar{\nabla}_a\bar{\nabla}_b\frac{r^{11}}{h}-(k+1)\frac{\Theta}{h}\sigma_k(r^{11})^2+\frac{\Theta}{h}(\sigma_k^{ij}\bar{g}_{ij})r^{11}\\
&-\frac{\Theta}{h} (r^{11})^2(2\sigma_k^{am}r^{nb}+\sigma_k^{ab,mn})\bar{\nabla}_1r_{ab}\bar{\nabla}_1r_{mn}\\
&-\frac{(r^{11})^2}{h}(2\bar{\nabla}_1\Theta\bar{\nabla}_1\sigma_k+\sigma_k\bar{\nabla}_1\bar{\nabla}_1\Theta)-\frac{\Theta r^{11}\sigma_k}{h^2}\\
&+2\frac{\Theta}{h}\sigma_k^{ab}\bar{\nabla}_ah\bar{\nabla}_b\frac{r^{11}}{h}+\frac{\Theta r^{11}}{h^2}\sigma_k^{ab}(r_{ab}-\bar{g}_{ab}h)+2\eta\frac{r^{11}}{h}.
\end{align*}
Balancing the terms gives
\begin{align*}
\left(\partial_\tau-\mathcal{L}-2\eta\right) \frac{r^{11}}{h}=&2\frac{\Theta}{h}\sigma_k^{ab}\bar{\nabla}_ah\bar{\nabla}_b\frac{r^{11}}{h}-(k+1)\frac{\Theta}{h}\sigma_k(r^{11})^2\\
&-\frac{\Theta}{h} (r^{11})^2(2\sigma_k^{am}r^{nb}+\sigma_k^{ab,mn})\bar{\nabla}_1r_{ab}\bar{\nabla}_1r_{mn}\\
&-\frac{(r^{11})^2}{h}(2\bar{\nabla}_1\Theta\bar{\nabla}_1\sigma_k+\sigma_k\bar{\nabla}_1\bar{\nabla}_1\Theta)+(k-1)\frac{\Theta\sigma_k r^{11}}{h^2}.
\end{align*}
To suitably group the terms on the right-hand side, note that
\begin{enumerate}
  \item
$f:=\sigma_k^{\frac{1}{k}}$ is inverse concave; therefore, by \cite{JU}*{(3.49)} we get
\begin{align}\label{inv conc}
(2f^{am}r^{bn}+f^{ab,mn})\bar{\nabla}_1r_{ab}\bar{\nabla}_1r_{mn}\geq 2\frac{(\bar{\nabla}_1 f)^2}{f}.
\end{align}
This in turn implies that
\begin{align*}(2\sigma_k^{am}r^{bn}+\sigma_k^{ab,mn})\bar{\nabla}_1r_{ab}\bar{\nabla}_1r_{mn}&\geq (k^2+k)f^{k-2}(\bar{\nabla}_1 f)^2\\
&=\frac{k+1}{k}\frac{(\bar{\nabla}_1 f^k)^2}{f^k}.
\end{align*}
  \item By the Schwartz inequality,\[2|\bar{\nabla}_1\Theta\bar{\nabla}_1f^k|\leq \frac{k+1}{k}\frac{\Theta(\bar{\nabla}_1 f^k)^2}{f^k}+\frac{k}{k+1}\frac{f^k(\bar{\nabla}_1 \Theta)^2}{\Theta}.\]
\end{enumerate}
Due to the preceding estimates, at $(u_{\tau},\tau)$ there holds
\begin{align*}
\partial_{\tau}\frac{r^{11}}{h}\leq-\frac{(r^{11})^2\sigma_k}{h}\left((k+1)\Theta+\bar{\nabla}_1\bar{\nabla}_1\Theta-\frac{k(\bar{\nabla}_1 \Theta)^2}{(k+1)\Theta}+(1-k)\frac{\Theta r_{11}}{h}\right)+\frac{2\eta r^{11}}{h}.
\end{align*}
Let $s$ be the arc-length of the great circle passing through $u_{\tau}$ with unit tangent vector $e_1.$
The sum of the first three terms in the bracket may be expressed as
\begin{align*}
(k+1)\Theta^{\frac{k}{k+1}}\left(\Theta^{\frac{1}{k+1}}+(\Theta^{\frac{1}{k+1}})_{ss}\right).
\end{align*}
Expand this last expression for $p+k-1>0$,
\begin{align*}
&\frac{\Theta^{\frac{1}{k+1}}+(\Theta^{\frac{1}{k+1}})_{ss}}{\Theta^{\frac{1}{k+1}}}\\
=&1+2\frac{2-p}{(k+1)^2}\frac{h_s\varphi_s}{h\varphi}
+\frac{(p-2)(p+k-1)}{(k+1)^2}\frac{h_s^2}{h^2}\\
&-\frac{k}{(k+1)^2}\frac{\varphi_s^2}{\varphi^2}+\frac{2-p}{k+1}\frac{h_{ss}}{h}+\frac{1}{k+1}\frac{\varphi_{ss}}{\varphi}\\
=& \frac{2-p}{k+1}\frac{h+h_{ss}}{h}+\frac{p-2}{(k+1)^2h\varphi}
\left(h_s\left(\frac{(p+k-1)\varphi}{h}\right)^{\frac{1}{2}}+\varphi_s\left(\frac{h}{(p+k-1)\varphi}\right)^{\frac{1}{2}}\right)^2\\
&+\frac{1}{\varphi(k+1)}\left\{\left(p+k-1\right)\varphi-\left(\frac{1}{(k+1)^2}\frac{p-2}{p+k-1}+\frac{k}{(k+1)^2}\right)\frac{\varphi_s^2}{\varphi}
+\varphi_{ss}\right\}\\
=& \frac{2-p}{k+1}\frac{h+h_{ss}}{h}+\frac{1}{k+1}\left\{\left(p+k-1\right)-\frac{p+k-2}{p+k-1}\left(\frac{\varphi_s}{\varphi}\right)^2
+\frac{\varphi_{ss}}{\varphi}\right\}\\
&+\frac{p-2}{(k+1)^2h\varphi}
\left(h_s\left(\frac{(p+k-1)\varphi}{h}\right)^{\frac{1}{2}}+\varphi_s\left(\frac{h}{(p+k-1)\varphi}\right)^{\frac{1}{2}}\right)^2.
\end{align*}
Let $p\geq 2$ and assume either of the following equivalent conditions hold
 \begin{align}\label{assum}
 \left\{
   \begin{array}{ll}
&(\varphi^{\frac{1}{p+k-1}})_{ss}+\varphi^{\frac{1}{p+k-1}}>0~\mbox{on every great circle},\\
&\bar{\nabla}_i\bar{\nabla}_j\varphi^{\frac{1}{p+k-1}}+\bar{g}_{ij}\varphi^{\frac{1}{p+k-1}}~\mbox{is positive definite}.
   \end{array}
 \right.
 \end{align}
Under this condition we conclude that
\begin{align*}
\partial_\tau\frac{r^{11}}{h}&\leq -\Theta\sigma_k\left(\frac{r^{11}}{h}\right)^2\left(c_{\varphi}h+(3-k-p)r_{11}\right)+2\eta\frac{r^{11}}{h}\leq-c_1\left(\frac{r^{11}}{h}\right)^2
+c_2\frac{r^{11}}{h},
\end{align*}
where $c_{\varphi}>0$ depends on the smallest eigenvalue of $\bar{\nabla}_i\bar{\nabla}_j\varphi^{\frac{1}{p+k-1}}+\bar{g}_{ij}\varphi^{\frac{1}{p+k-1}}$ with respect to $\bar{g}$ and we used the lower and upper bounds on $h,\sigma_k,\eta$ from Lemmas  \ref{eta}, \ref{key lemma}.
By the maximum principle
$r^{11}\leq L$
for some $L$ depending on $M_0,\varphi.$
Thus the principal radii of curvature satisfy
$\frac{1}{L}\leq\lambda_i.$
To finish the proof, note that $\frac{1}{L^{k-1}}(\max\lambda_i)\leq \sigma_k\leq d.$

\emph{Second proof}. Arrange the principal radii of curvature as $\lambda_1\leq\cdots\leq\lambda_n.$
We show that $h\lambda_1$ satisfies a suitable differential inequality in a viscosity sense.

Let us fix a point $(u_{\tau},\tau)$ with $\tau>0$ and suppose at this point the multiplicity of $\lambda_1$ is $\mu$; that is,
$\lambda_1=\cdots=\lambda_{\mu}<\lambda_{\mu+1}\leq\cdots\leq\lambda_n.$
Choose an orthonormal frame for $T_{u_{\tau}}\mathbb{S}^{n}$ such that
$r_{ij}=\lambda_i\delta_{ij},~\bar{g}_{ij}=\delta_{ij}.$

Let $\xi$ be an arbitrary $C^2$ lower support of $h\lambda_1$ at $(u_{\tau},\tau)$. That is, for some $\varepsilon>0$ and an open neighborhood $\mathcal{O}_{u_{\tau}}$ of $u_{\tau}$, we have $\xi(u,t)\leq (h\lambda_1)(u,t)$ for all $(u,t)\in \mathcal{O}_{u_{\tau}}\times (\tau-\varepsilon,\tau]$ and $\xi(u_{\tau},\tau)=(h\lambda_1)(u_{\tau},\tau).$ With similar calculations as in \cite{brd}*{Lemma 5} at $(u_{\tau},\tau)$ we obtain
\begin{itemize}
  \item $\bar{\nabla}_i\bar{\nabla}_i\xi\leq \bar{\nabla}_i\bar{\nabla}_ihr_{11}-2h\sum_{b>\mu}\frac{(\bar{\nabla}_ir_{1b})^2}{\lambda_b-\lambda_1},$
  \item $r_{kl}\bar{\nabla}_ih+h\bar{\nabla}_ir_{kl}=\delta_{kl}\bar{\nabla}_i\xi$ for all $1\le k,l\leq \mu.$
\end{itemize}
In addition, note that
\[\forall (u,t)\in \mathcal{O}_{u_{\tau}}\times (\tau-\varepsilon,\tau]:\quad h\frac{r_{1k}\bar{g}^{kl}r_{l1}}{r_{11}}\geq h\lambda_1\geq \xi.
\]
In fact, assume the hypersurface is given as an embedding of $\mathbb{S}^n$ via the inverse Gauss map. Then the second fundamental form is $r_{ij}$. By \cite{ChoiQk}*{Proposition 4.1}, in any chart we have $\frac{r_{11}}{g_{11}}\leq \frac{1}{\lambda_1}$ where $g_{ij}$ is the metric of the hypersurface. Due to the identity $g_{ii}=r_{ik}\bar{g}^{kl}r_{li}$, we obtain $h\frac{r_{1k}\bar{g}^{kl}r_{l1}}{r_{11}}\geq h\lambda_1\geq \xi$. Moreover, since this last inequality becomes an equality at $(u_{\tau},\tau)$, we get
\[\partial_\tau(hr_{11})\big|_{(u_{\tau},\tau)}=\partial_\tau\left(h\frac{r_{1k}\bar{g}^{kl}r_{l1}}{r_{11}}\right)\big|_{(u_{\tau},\tau)}\leq \partial_{\tau}\xi\big|_{(u_{\tau},\tau)}.\]
Putting all these facts together yields
\begin{align*}
(\partial_\tau-\mathcal{L})\xi
\geq&h(\partial_\tau-\mathcal{L})r_{11}+\xi(\partial_\tau-\mathcal{L})h
+2h\Theta\sigma_k^{aa}\sum_{b>\mu}\frac{(\bar{\nabla}_ar_{1b})^2}{\lambda_b-\lambda_1}\\
&+2\Theta\sigma_k^{aa}\frac{(\bar{\nabla}_a\xi-\lambda_1\bar{\nabla}_ah)^2}{\xi}-2\Theta\sigma_k^{ab}\bar{\nabla}_a\xi\bar{\nabla}_b\log \frac{\xi}{h}.
\end{align*}
Then owing to Lemma \ref{ev equ} and (\ref{inv conc}), we arrive at the estimate
\begin{align*}
&(\partial_\tau-\mathcal{L})\xi+2\Theta\sigma_k^{ab}\bar{\nabla}_a\xi\bar{\nabla}_b\log \frac{\xi}{h}+((k-1)\varphi h^{1-p}\sigma_k +2\eta)\xi\\
\geq& 2h\Theta\left(\sigma_k^{aa}\sum_{b>\mu}\frac{(\bar{\nabla}_ar_{1b})^2}{\lambda_b-\lambda_1}-\sigma_k^{aa}r^{bb}(\bar{\nabla}_1r_{ab})^2
+\sigma_k^{aa}\frac{(\bar{\nabla}_a\xi-\lambda_1\bar{\nabla}_ah)^2}{\lambda_1h^2}\right)\\
&+(k+1)h\Theta\sigma_k
+2h\bar{\nabla}_1\Theta\bar{\nabla}_1\sigma_k
+h\sigma_k\bar{\nabla}_1\bar{\nabla}_1\Theta+\frac{k+1}{k}h\Theta\frac{(\bar{\nabla}_1 \sigma_k)^2}{\sigma_k}.
\end{align*}
By Schwartz's inequality we obtain
\begin{align*}
&(\partial_\tau-\mathcal{L})\xi+2\Theta\sigma_k^{ab}\bar{\nabla}_a\xi\bar{\nabla}_b\log \frac{\xi}{h}+((k-1)\varphi h^{1-p}\sigma_k +2\eta)\xi\\
&\geq 2h\Theta\left(\sigma_k^{aa}\sum_{b>\mu}\frac{(\bar{\nabla}_ar_{1b})^2}{\lambda_b-\lambda_1}-\sigma_k^{aa}r^{bb}(\bar{\nabla}_1r_{ab})^2
+\sigma_k^{aa}\frac{(\bar{\nabla}_a\xi-\lambda_1\bar{\nabla}_ah)^2}{\lambda_1h^2}\right)\\
&+h\sigma_k\left((k+1)\Theta
+\bar{\nabla}_1\bar{\nabla}_1\Theta-\frac{k}{k+1}\frac{(\bar{\nabla}_1 \Theta)^2}{\Theta}\right).
\end{align*}
We show that
\[\mathcal{R}:=\sigma_k^{aa}\sum_{b>\mu}\frac{(\bar{\nabla}_ar_{1b})^2}{\lambda_b-\lambda_1}-\sigma_k^{aa}r^{bb}(\bar{\nabla}_1r_{ab})^2
+\sigma_k^{aa}\frac{(\bar{\nabla}_a\xi-\lambda_1\bar{\nabla}_ah)^2}{\lambda_1h^2}~\hbox{is non-negative}.\]
To see this, note that we can estimate the second term in $\mathcal{R}$ as follows
\begin{align*}
r^{bb}(\bar{\nabla}_1r_{ab})^2=r^{bb}(\bar{\nabla}_ar_{1b})^2&=\sum_{b\leq \mu}r^{bb}(\bar{\nabla}_ar_{1b})^2+\sum_{b>\mu}r^{bb}(\bar{\nabla}_ar_{1b})^2\\
&=\frac{(\bar{\nabla}_a\xi-\lambda_1\bar{\nabla}_ah)^2}{\lambda_1h^2}+\sum_{b>\mu}\frac{1}{\lambda_b}(\bar{\nabla}_ar_{1b})^2.
\end{align*}
We may continue as in the final part of the first proof to deduce that at $(u_{\tau},\tau),$
\begin{align}\label{df ineq}
\partial_\tau \xi\geq \mathcal{L}\xi-2\Theta\sigma_k^{ab}\bar{\nabla}_a\xi\bar{\nabla}_b\log \frac{\xi}{h}-((p+k-3)\varphi h^{1-p}\sigma_k +2\eta)\xi+h^{3-p}\sigma_kc_{\varphi}
\end{align}
for a positive constant depending on the $C^2$ norm of $\varphi.$ Now suppose for the sake of contradiction that $h\lambda_1$ for the first time $\tau>0$ at $u_{\tau}$ equals to
\[\xi:=\frac{1}{2}\min\left\{\frac{\min (h^{3-p}\sigma_k)c_{\varphi}}{\max\left((p+k-3)\varphi h^{1-p}\sigma_k +2\eta\right)},\min(h\lambda_1(\cdot,0))\right\}.\]
So $\xi$ serves as a lower support for $h\lambda_1$ on $M^n\times [0,\tau]$. But (\ref{df ineq}) yields a contradiction.
\end{proof}
\begin{remark}\label{rem}
\item[(1)] Since mixed volumes are monotonic increasing in each argument \cite{Schneider}*{page 282},
\[h_{\min}^{k+1}\leq\frac{\int_{\mathbb{S}^n}h\sigma_kdx}{\omega_n}\leq h_{\max}^{k+1}.\]
If we have lower and upper estimates for $h(\cdot,\tau)$, then we have lower and upper bounds on $\eta(\tau).$
\item[(2)] From the proof of the previous lemma and the first remark it is clear that if $p\geq 2$ and one can establish lower and upper bounds for $h,\sigma_k$ along the normalized flow, then under the assumption (\ref{assum}) there are lower and upper bounds on the principal curvatures.
\item[(3)] For $p\geq k+1$ we could avoid using the result of \cite{Ger} to prove Lemma \ref{eta}. In fact, since $\bar{\nabla}_i\bar{\nabla}_jh$ is non-negative at the maximum point of $h$, by the monotonicity of the entropy we have
\begin{align*}
\partial_{\tau}h_{\max}\leq h_{\max}(\varphi_{\max} h_{\max}^{1+k-p} -\eta(\tau)) \leq h_{\max}(\varphi_{\max} h_{\max}^{1+k-p}-\eta(0))
\end{align*}
in the sense of the lim sup of forward difference quotients; see, \cite{ham}.
Since $p>k+1$, if $h_{\max}$ becomes too large, then the right-hand side will be negative. So $h$ remains bounded above by some constant $b$.
To get a lower bound on $h$, note that
$ \eta(\tau)\leq b^{k+1};$
therefore,
\begin{align*}
\partial_{\tau}h_{\min}\geq h_{\min}\left(\varphi_{\min} h_{\min}^{1+k-p} -\eta(\tau)\right)\geq h_{\min}\left(\varphi_{\min} h_{\min}^{1+k-p} -b^{k+1}\right).
\end{align*}
If $h_{\min}$ becomes strictly less than a critical value, then the right-hand side will be positive; therefore, $h(\cdot,\tau)\geq a$ for some positive constant $a.$ So we have shown for $p>k+1,$
$a^{k+1}\leq \eta\leq b^{k+1}.$
For $p=k+1$, the proof of Lemma \ref{grad estimate} does not need any control on $\eta$.
\item[(4)] If $k=1$, in the proof of the previous lemma we could apply the maximum principle directly to $r_{11}$ instead of $hr_{11}.$
\end{remark}
To obtain lower and upper bounds on $\sigma_k$ without the limitation set by homogeneity degree of the speed, we will need to use auxiliary functions that are not homogeneous. The next lemma gives a lower bound on $\sigma_k$ in a large generality. The proof does not use any particular structure of $\sigma_k.$
\begin{lemma}\label{lem lower speed bound}
Let $h(\cdot,\tau)$ be a solution to the flow (\ref{nf}) such that $a\leq h(\cdot,\tau)\leq b$ for some positive constants $a,b.$ Then $\sigma_k(\cdot,\tau)\geq c$ for a positive constant $c$.
\end{lemma}
\begin{proof}
Let $A>0$ be a constant to be determined later. We will apply the maximum principle to the following auxiliary function considered in \cite{HS}, \[\chi:=\log(\Theta\sigma_k)-A\frac{\rho^2}{2}.\]
Owing to Lemma \ref{ev equ} we have
\begin{align*}
\left(\partial_\tau-\mathcal{L}\right)\log(\varphi h^{2-p}\sigma_k)=&\Theta\sigma_k^{ab}\bar{\nabla}_a \log(\varphi h^{2-p}\sigma_k)\bar{\nabla}_b \log(\varphi h^{2-p}\sigma_k)\\
&+(2-p)\varphi h^{1-p}\sigma_k+\varphi h^{2-p}\sigma_k^{ij}\bar{g}_{ij}-(2+k-p)\eta.
\end{align*}
Consequently using Lemma \ref{ev equ} the evolution equation of $\chi$ reads as
\begin{align*}
\left(\partial_\tau-\mathcal{L}\right)\chi
=&\Theta\sigma_k^{ab}\bar{\nabla}_a \log(\varphi h^{2-p}\sigma_k)\bar{\nabla}_b \log(\varphi h^{2-p}\sigma_k)\\
&+(2-p)\varphi h^{1-p}\sigma_k+\varphi h^{2-p}\sigma_k^{ij}\bar{g}_{ij}-(2+k-p)\eta\\
&-A(k+1)h\Theta\sigma_k+A\eta \rho^2-A\sigma_k\bar{g}^{ij}\bar{\nabla}_ih\bar{\nabla}_j\Theta+A\Theta\sigma_k^{ab}r_a^mr_{mb}.
\end{align*}
Dropping some positive terms and rearranging terms yield
\begin{align*}
\left(\partial_\tau-\mathcal{L}\right)\chi
\geq&\left(\frac{A}{2}\eta\rho^2+(2-p)\frac{e^{\chi+A\frac{\rho^2}{2}}}{h}-(2+k-p)\eta\right)\\
&+A\Theta\sigma_k\left(\frac{\eta \rho^2}{2e^{\chi+A\frac{\rho^2}{2}}}-\bar{g}^{ij}\bar{\nabla}_ih\bar{\nabla}_j\log\Theta-(k+1)h\right).
\end{align*}
Choose $A>\max\frac{2}{\rho^2}(2+k-p)$. Thus if $\chi$ becomes very negative then the right-hand side becomes positive; this is due to the uniform upper bound on $|\bar{\nabla}_ah(\cdot,\tau)|$ and lower bound on $\eta;$ see Remark \ref{rem}.
\end{proof}
The next lemma gives an upper bound on $\sigma_k(\cdot,\tau)$ for every $p,k$. The proof employs the following property of $\sigma_k$ (due to its inverse concavity; see, e.g., \cite{AMZ}):
\begin{align}\label{inver concavity}
\sigma_k^{ab}r_a^mr_{mb}\geq k\sigma_k^{1+\frac{1}{k}}.
\end{align}
\begin{lemma}\label{upper speed bound}
Let $h(\cdot,\tau)$ be a solution to the flow (\ref{nf}) such that $\varepsilon\leq \rho^2(\cdot,\tau)\leq \frac{1}{\varepsilon}$ for an $0<\varepsilon<1.$ Then $\sigma_k(\cdot,\tau)\leq d$ for some positive constant $d$.
\end{lemma}
\begin{proof}
We will apply the maximum principle to the auxiliary function \[\chi(\cdot,\tau):=\frac{\varphi h^{1-p}\sigma_k}{1-\varepsilon\frac{\rho^2}{2}}(\cdot,\tau).\] Meanwhile note that for two positive smooth functions $f,g:\mathbb{S}^n\times [0,T)\to \mathbb{R}$,
\[(\partial_{t}-\mathcal{L})\frac{f}{g}= \frac{1}{g}(\partial_{t}-\mathcal{L})f-\frac{f}{g^2}(\partial_{t}-\mathcal{L})g+2\Theta\sigma_k^{ij}\bar{\nabla}_i\log g\bar{\nabla}_j\frac{f}{g}. \]
Therefore in view of the evolution equation (\ref{ev of speed divided by h}) and Lemma \ref{ev equ}, at the point where $\chi$ attains its maximum we have
\begin{align}\label{time derv chi}
\partial_{\tau} \chi\leq& \frac{1}{1-\varepsilon\frac{\rho^2}{2}}\Biggl(2\varphi h^{1-p}\sigma_k^{ij}\bar{\nabla}_i(\varphi h^{1-p}\sigma_k)\bar{\nabla}_jh
\\
&+(1+k-p)\left(\varphi h^{1-p}\sigma_k\right)^2+(p-k-1)\eta\varphi h^{1-p}\sigma_k\Biggr)\nonumber\\
&+\frac{\varepsilon \varphi h^{1-p}\sigma_k}{(1-\varepsilon\frac{\rho^2}{2})^2}\left((k+1)h\Theta\sigma_k-\eta\rho^2
+\sigma_k\bar{g}^{ij}\bar{\nabla}_ih\bar{\nabla}_j\Theta-\Theta\sigma_k^{ab}r_a^mr_{mb}\right)\nonumber
\end{align}
and
\begin{align}\label{grad extra bonus}
\bar{\nabla}_i(\varphi h^{1-p}\sigma_k)=-\varepsilon\rho\frac{\varphi h^{1-p}\sigma_k\bar{\nabla}_i\rho}{1-\varepsilon\frac{\rho^2}{2}}=-
\varepsilon\frac{\varphi h^{1-p}\sigma_kr_{i}^m\bar{\nabla}_mh}{1-\varepsilon\frac{\rho^2}{2}}.
\end{align}
By Remark \ref{rem}, $\eta$ is bounded above. Putting (\ref{grad extra bonus}) into (\ref{time derv chi}), using (\ref{inver concavity}) and the lower and upper bounds on $\rho,h,\Theta$ and $|\bar{\nabla}\Theta|,|\bar{\nabla}h|$, we find that there exist positive constants $c_1,c_2,c_3$ such that
\[\partial_{\tau} \chi\leq c_1\chi+c_2\chi^2-c_3\chi^{2+\frac{1}{k}}.\]
The maximum principle completes the proof.
\end{proof}
\begin{remark}
The auxiliary function we considered in the previous lemma is very robust. Consider a curvature flow $\partial
_th=\varphi h^{\alpha}f^p,$
where $p>0$ and $f$ is a $1$-homogeneous function of principal radii of curvature satisfying $f^{ij}r_i^kr_{kj}\geq cf^2$ for some $c>0$ and $\varphi$ is a positive smooth function defined on the unit sphere. In particular, if $f$ is inverse concave or we are privileged with a pinching estimate the inequality holds. Then in the presence of lower and upper bounds on the support function we can apply the maximum principle to $\frac{\varphi h^{\alpha-1}f}{1-\varepsilon \rho^2}$ (for a suitable $\varepsilon>0$) to obtain a uniform upper bound on the speed.
\end{remark}
\section{convergence of the normalized solution}
In this section we complete the proofs of our main theorems.

$\bullet$ Theorem \ref{main theorem1}: uniform $C^2$ regularity estimates were obtained in previous section.

$\bullet$ Theorem \ref{main theorem1a}: uniform lower and upper bounds on the support function and $\sigma_n(\cdot,\tau)$ follow from Lemma \ref{key lemma}. Moreover, uniform $C^2$ regularity estimates were proved in \cite{BIS}*{Lemma 8}.

$\bullet$ Theorem \ref{main theorem2}: the lower and upper bounds on $h(\cdot,\tau)$ follow from the strong gradient estimates of Chow-Gulliver \cite{CHGU}.
Therefore, by Remark \ref{rem}-(1), $\eta(\tau)$ is controlled from above and below. Now the lower and upper bounds on $\sigma_k(\cdot,\tau)$ follow from Lemmas \ref{lem lower speed bound} and \ref{upper speed bound}. Then by Remark \ref{rem}-(2), we can deduce uniform lower and upper bounds on the principal curvatures $\kappa_i(\cdot,\tau)$.

Now in all cases higher order regularity estimates follow from Krylov and Safonov \cite{Kryl} and Schauder theory. The convergence of the normalized solution for a subsequence of times to a soliton then follows from monotonicity of $\mathcal{A}_{k,p}^{\varphi}[h(\cdot,\tau)]$ established in Lemma \ref{monotonicity}. The convergence for the \emph{full} sequence of the normalized solution follows from the uniqueness result of \cite{Lut1}; see also \cite{HMS}*{page 149} for another proof of uniqueness.\footnote{For $p=k+1$, the uniqueness of a strictly convex solution to (\ref{soliton}) is upto dilations. Since here we are dealing with normalized solutions of the flow such that $\int_{\mathbb{S}^n}\frac{h^{k+1}(x,\tau)}{\varphi(x)}dx$ is constant, the limit is unique.}

\section{Loss of smoothness}
\begin{example}
Suppose $p+k-1>0$ and $k<n$. There exist a rotationally symmetric positive function $\varphi\in C^2(\mathbb{S}^n)$ satisfying $((\varphi^{\frac{1}{p+k-1}})_{\theta\theta}+\varphi^{\frac{1}{p+k-1}})\big|_{\theta=0}<0$ and a smooth, closed, strictly convex initial hypersurface $M_0$ such that the solution to the flow (\ref{F-param}) will lose smoothness.
\end{example}
\begin{proof}
We follow the same approach as in \cite{AMZ}*{Corollary 1}. Let $0<h(\cdot,0)\in C^{\infty}(\mathbb{S}^n)$ be a rotationally symmetric support function (e.q., non-negative spherical hessian) such that $$h_{\theta\theta}+h\geq 0,\quad\sigma_k(\bar{\nabla}_i\bar{\nabla}_jh(\cdot,0)+\bar{g}_{ij}h(\cdot,0))>0.$$

Since $\sigma_k(\cdot,0)>0$, a rotationally symmetric solution to the following equation exists for a short time,
\begin{align}\label{exmaple}
h&:\mathbb{S}^n\times[0,T)\to \mathbb{R}\\
\partial_t h(\cdot,t)&=\varphi h^{2-p}\sigma_k(\bar{\nabla}_i\bar{\nabla}_jh+\bar{g}_{ij}h).\nonumber
\end{align}
The eigenvalues of $\bar{\nabla}_i\bar{\nabla}_jh+\bar{g}_{ij}h$ with respect to $\bar{g}$ are given by
\[\zeta_1=h_{\theta\theta}+h,\quad \zeta_2=\cdots=\zeta_n=h-\tan(\theta)h_{\theta}.\]
Also, note that
\[(\zeta_2)_{\theta}=\tan(\theta)(\zeta_2-\zeta_1),\quad (\zeta_2)_{\theta\theta}=(1+2\tan^2(\theta))(\zeta_2-\zeta_1)-\tan(\theta)(\zeta_1)_{\theta}.\]
 From the definition of $\zeta_{1}$, we obtain
\begin{align*}
\partial_t\zeta_1
=&\varphi h^{2-p}\frac{\partial\sigma_k}{\partial\zeta_i}(\zeta_i)_{\theta\theta}+\varphi h^{2-p}\frac{\partial^2\sigma_k}{\partial\zeta_i\partial\zeta_j}(\zeta_i)_{\theta}(\zeta_j)_{\theta}+\varphi h^{2-p}\sigma_k\\
&+2(\varphi h^{2-p})_{\theta}(\sigma_k)_{\theta}+\sigma_k(\varphi h^{2-p})_{\theta\theta}.
\end{align*}
Thus for the particular choice $\theta=0,$ we have
\begin{align*}
(\partial_t\zeta_1)\big|_{(0,t)}
=&\varphi h^{2-p}\frac{\partial\sigma_k}{\partial\zeta_1}(\zeta_1)_{\theta\theta}+(n-1)\varphi h^{2-p}\frac{\partial\sigma_k}{\partial\zeta_2}(\zeta_2-\zeta_1)+\varphi h^{2-p}\sigma_k\\
&+2\frac{\partial\sigma_k}{\partial\zeta_1}(\varphi h^{2-p})_{\theta}(\zeta_1)_{\theta}+\sigma_k(\varphi h^{2-p})_{\theta\theta}\\
=&\varphi h^{2-p}\frac{\partial\sigma_k}{\partial\zeta_1}(\zeta_1)_{\theta\theta}-\varphi h^{2-p}((n-1)\frac{\partial\sigma_k}{\partial\zeta_2}+\frac{\partial\sigma_k}{\partial\zeta_1})\zeta_{1}\\
&+2\frac{\partial\sigma_k}{\partial\zeta_1}(\varphi h^{2-p})_{\theta}(\zeta_1)_{\theta}+\sigma_k(\varphi h^{2-p})_{\theta\theta}+(k+1)\varphi h^{2-p}\sigma_k.
\end{align*}
Let $r$ be a $\pi$-periodic, $C^{\infty}$ function such that $r(\theta)=r(-\theta)$, it is zero on $[-\frac{\pi}{4},\frac{\pi}{4}]$ and positive elsewhere in $[-\frac{\pi}{2},\frac{\pi}{2}]$. Now define
\[h(\theta,0):=\sin(\theta)\int_0^{\theta}r(\alpha)\cos(\alpha)d\alpha+\cos(\theta)\int_{\theta}^{\frac{\pi}{2}}r(\alpha)\sin(\alpha)d\alpha.\]
Note that $(h(\cdot,0))_{\theta}\big|_{\theta=0}=0$ and for all $\theta\in[-\frac{\pi}{4},\frac{\pi}{4}]$ we have \[\zeta_1(0,0)=(\zeta_1)_{\theta}(0,0)=(\zeta_1)_{\theta\theta}(0,0)=0.\] Hence we obtain
\begin{align*}
(\partial_t\zeta_1)\big|_{(0,0)}=\left(\sigma_k h^{2-p}(\varphi_{\theta\theta}+(p+k-1)\varphi)\right)\big|_{(0,0)}.
\end{align*}
Pick any positive function such that
\[\varphi_\theta(0)=0,\quad((\varphi^{\frac{1}{p+k-1}})_{\theta\theta}+\varphi^{\frac{1}{p+k-1}})\big|_{\theta=0}<0.\]
For example, $\varphi(\theta)=(\cos^2(\theta)+\frac{1}{2})^{p+k-1}.$ So $\zeta_{1}(0,t)$ becomes negative for $t>0$ sufficiently small.

Since the solution to (\ref{exmaple}) depends continuously on the initial data, the nearby smooth, closed strictly convex hypersurfaces will lose smoothness.
Since $\sigma_n(0,0)=0$, the argument fails if $k=n$, as expected in view of the results in \cite{BIS} and also Theorem \ref{main theorem1a} here.
\end{proof}
\begin{remark}
It would be interesting to improve Theorem \ref{main theorem1} by allowing \[\bar{\nabla}_i\bar{\nabla}_j\varphi^{\frac{1}{p+k-1}}+\bar{g}_{ij}\varphi^{\frac{1}{p+k-1}}\] to be non-negative definite. We conclude the paper with the following questions.
\begin{Qu}
Is it possible to obtain a gradient bound for the support function of the normalized flow for $k<n, 1<p<k+1$ in the class of origin-symmetric solutions?
\end{Qu}
\begin{Qu}
If $k<n, p<2, \varphi\equiv1,$ what can be said about the asymptotic behavior of the flow?
\end{Qu}
\end{remark}
\bibliographystyle{amsplain}

\begin{thebibliography}{999}
\bibitem{BA1} B. Andrews, \textit{Monotone quantities and unique limits for evolving convex hypersurfaces}, Int. Math. Res. Not. IMRN 20(1997):1001-1031.
\bibitem{AMZ} B. Andrews, J. McCoy, Y. Zheng, \textit{Contracting convex hypersurfaces by curvature}, Calc. Var. Partial Differential Equations 47(2013):611-665.
\bibitem{BIS} P. Bryan, M.N. Ivaki, J. Scheuer, \textit{A unified flow approach to smooth, even $L_p$-Minkowski problems}, arXiv preprint arXiv:1608.02770 (2016).
\bibitem{Berg} C. Berg, \textit{Corps convexes et potentiels sph\'{e}riques}, (French), Mat.-Fys. Medd. Danske Vid.Selsk. 37(1969).
\bibitem{CBC} G. Bianchi, K.J. B\"{o}r\"{o}czky, A. Colesanti, \textit{Smoothness in the $L_p$ Minkowski problem for $p<1$}, arXiv preprint arXiv:1706.06310 (2017).
\bibitem{brd} S. Brendle, K. Choi, P. Daskalopoulos, \textit{Asymptotic behavior of flows by powers of the Gaussian curvature}, to appear in Acta Math., arXiv:1610.08933 (2016).
\bibitem{Chengyau} S.Y. Cheng, and S.T. Yau, \textit{On the regularity of the solution of the $n$-dimensional Minkowski problem}, Comm. Pure Appl. Math. 29(1976):495-516.
\bibitem{ChoiQk} K. Choi, P. Daskalopoulos, \textit{The $Q_k$ flow on complete non-compact graphs}, arXiv preprint arXiv:1603.03453 (2016).
\bibitem{CHGU} B. Chow, R. Gulliver, \textit{Aleksandrov reflection and nonlinear evolution equations, I: The n-sphere and n-ball}, Calc. Var. Partial Differential Equations 4(1996):249-264.
\bibitem{BT} B. Chow, D.H. Tsai, \textit{Expansion of convex hypersurfaces by nonhomogeneous functions of curvature}, Asian J. Math. 1(1997):769-784.
\bibitem{chwang} K.S. Chou, and X.-J. Wang, \textit{The $L_p$-Minkowski problem and the Minkowski problem in centroaffine geometry}, Adv. in Math. 205(2006):33-83.
\bibitem{ivakitrans} M.N. Ivaki, \textit{The planar Busemann-Petty centroid inequality and its stability}, Trans. Amer. Math. Soc. 368(2016): 3539-3563.
\bibitem{ivakifunc} M.N. Ivaki, \textit{Deforming a hypersurface by Gauss curvature and support function}, J. Funct. Anal. 271(2016):2133-2165.
\bibitem{Fi1} W.J. Firey, \textit{The determination of convex bodies from their mean radius of curvature functions}, Mathematika 14(1967):1-14.
\bibitem{Fi2} W.J. Firey, \textit{Christoffel problems for general convex bodies}, Mathematik 15(1968):7-21.
\bibitem{Fi3} W.J. Firey, \textit{Intermediate Christoffel-Minkowski problems for figures of revolution}, Israel J. Math. 8(1970):384-390.
\bibitem{Fillmore} P.J. Fillmore, \textit{Symmetries of surfaces of constant width}, J. Differ. Geom. 3(1969): 103--110.
\bibitem{Ger} C. Gerhardt, \textit{Non-scale-invariant inverse curvature flows in Euclidean space}, Cal. Var. Partial Differential Equations 49(2014):471-489.
\bibitem{GYY} P. Goodey, V. Yaskin, M. Yaskina, \textit{A Fourier transform approach to Christoffel's problem}, Trans. Amer. Math. Soc. 363(2011):6351-6384.
\bibitem{Guan-Ma} P. Guan, X.-N. Ma, \textit{Christoffel-Minkowski problem I: convexity of solutions of a hessian equation}, Invent. Math. 151(2003):553-577.
\bibitem{GuanXia} P. Guan, C. Xia, \textit{$L^p$ Christoffel-Minkowski problem: the case $1<p<k+1$}, Cal. Var. Partial Differential Equations (2018) 57:69  https://10.1007/s00526-018-1341-y.
\bibitem{ham} R.S. Hamilton, \textit{Four-manifolds with positive curvature operator}, J. Differential Geom. 24(1986):153-179.
\bibitem{Kryl} N.V Krylov, and M.V. Safonov, \textit{Certain properties of parabolic equations with measurable coefficients}, Izv. Akad. Nauk SSSR Ser. Mat. 40(1981):161-175; English transl., Math. USSR Izv. 16(1981):151-164.
\bibitem{HMS} C. Hu, X.-N Ma, C. Shen, \textit{On the Christoffel-Minkowski problem of Firey's $p$-sum}, Cal. Var. Partial Differential Equations 21(2004):137-155.
\bibitem{Lut1} E. Lutwak, \textit{The Brunn-Minkowski-Firey theory I: Mixed volumes and the Minkowski problem}, J. Differential Geom. 38(1993):131-150.
\bibitem{Lut2} E. Lutwak, \textit{The Brunn-Minkowski-Firey theory II: Affine and geominimal surface areas}, Adv. in Math. 118(1996):244-294.
\bibitem{Lut3} E. Lutwak, V. Oliker, \textit{On the regularity of solutions to a generalization of the Minkowski problem}, J. Differential Geom. 41(1995):227-246.
\bibitem{LYZ} E. Lutwak, D. Yang, G. Zhang, \textit{Sharp affine $L_p$ Sobolev inequalities}, J. Differential Geom.  62(2002) 17--38.
\bibitem{HS} H. Kr\"{o}ner, J. Scheuer, \textit{Expansion of pinched hypersurfaces of the Euclidean and hyperbolic space by high powers of curvature}, arXiv preprint arXiv:1703.07087 (2017).
\bibitem{pog} A.V. Pogorelov, \textit{On the question of the existence of a convex surface with a given sum principal radii of curvature}, (in Russian). Usp. Mat. Nauk 8(1953):127-130.
\bibitem{JU} J.I.E. Urbas, \textit{An expansion of hypersurfaces}, J. Differential Geom. 91(1991):91-125.
\bibitem{JU1} J.I.E. Urbas, \textit{On the expansion of starshaped hypersurfaces by symmetric functions of their principal curvatures}, Math. Z. 205(1990): 355-372.
\bibitem{Schneider} R. Schneider, \textit{Convex bodies: the Brunn-Minkowski theory}, Second edition, No. 151. Cambridge University Press, 2013.
\end{thebibliography}

\end{document}